\documentclass[a4paper, 12pt]{article}

\usepackage[sort&compress]{natbib}
\bibpunct{(}{)}{;}{a}{}{,} 

\usepackage{amsthm, amsmath, amssymb, mathrsfs, multirow, url, subfigure}
\usepackage{graphicx} 
\usepackage{ifthen} 
\usepackage{amsfonts}
\usepackage[usenames]{color}

\usepackage[margin=1in]{geometry}
\usepackage{fancyhdr}
\theoremstyle{plain} 
\newtheorem{thm}{Theorem}

\newtheorem{lem}{Lemma}

\theoremstyle{definition}

\theoremstyle{remark}

\newcommand{\prob}{\mathsf{P}} 
\newcommand{\E}{\mathsf{E}}
\newcommand{\var}{\mathsf{V}}

\newcommand{\nm}{{\sf N}}

\newcommand{\chisq}{{\sf ChiSq}}

\newcommand{\RR}{\mathbb{R}}

\newcommand{\eps}{\varepsilon}

\DeclareMathOperator{\logit}{logit}

\newcommand{\tR}{\widetilde{R}}

\title{Asymptotically optimal inference in sparse sequence models with a simple data-dependent measure\footnote{This work is partially supported by the U.S.~National Science Foundation, DMS--1811802.}}
\author{
Ryan Martin\footnote{Department of Statistics, North Carolina State University, {\tt rgmarti3@ncsu.edu}}
}
\date{\today}

\begin{document}

\maketitle 

\thispagestyle{fancy}

\begin{abstract}
For high-dimensional inference problems, statisticians have a number of competing interests.  On the one hand, procedures should provide accurate estimation, reliable structure learning, and valid uncertainty quantification.  On the other hand, procedures should be computationally efficient and able to scale to very high dimensions.  In this note, I show that a very simple data-dependent measure can achieve all of these desirable properties simultaneously, along with some robustness to the error distribution, in sparse sequence models. 

\smallskip

\end{abstract}

\section{Introduction}
\label{S:intro}

Dating back at least to \citet{stein1956, stein1981}, a fundamental problem in statistical inference is that of estimating a high-dimensional mean vector under additive noise.  Specifically, suppose that the observable data is $Y^n = (Y_1,\ldots,Y_n)^\top$ with posited model 
\begin{equation}
\label{eq:model}
Y_i = \theta_i + Z_i, \quad i=1,\ldots,n,
\end{equation}
where the errors, $Z_1,\ldots,Z_n$, are independent and identically distributed (iid) with mean zero and subgaussian tails; see Section~\ref{S:properties} for specifics.  While I don't assume the error distribution to be normal, I do assume that it's fully {\em known}, with no parameters to be estimated.  The goal is inference on $\theta=(\theta_1,\ldots,\theta_n)^\top$. 

In practical applications using models like in \eqref{eq:model}, e.g., \citet{efron2010book}, the index $n$ is generally quite large, which makes this a {\em high-dimensional} inference problem.  As is often the case in such problems, certain structural assumption on $\theta = (\theta_1,\ldots,\theta_n)^\top$ are necessary in order to achieve accurate estimation.  Here, the structural assumption I'll consider is {\em sparsity}---most of the entries in the $\theta$ vector are zero.  This notion of sparsity is consistent with the science: in genomics applications, the biology says that only a relatively small number of genes would be associated with a particular phenotype.  Mathematically weaker notions of sparsity are possible, ones that don't assume any exact zeros, which I discuss below in Section~\ref{S:discuss}.  

There now many different approaches to this problem.  On the non-Bayesian side, the focus has been on developing shrinkage and thresholding estimators, and notable references include \citet{james.stein.1961}, \citet{efronmorris1973}, \citet{donohojohnstone1994b}, and \citet{abramovich2006}, among others.  On the Bayesian side, the strategy is to construct suitable sparsity-inducing priors.  This can be done using continuous shrinkage priors \citep[e.g.,][]{dunson.shrinkage, bhadra.hsplus.2017, carvalho.polson.scott.2010}, or with spike-and-slab priors \citep[e.g.,][]{castillo.vaart.2012, johnstonesilverman2004, martin.walker.eb}.  Inspired by the developments in \citet{belitser.ddm} and \citet{belitser.nurushev.uq}, here I want to take a different approach.  In particular, I construct a so-called {\em data-dependent measure} which is based on a simple approximation to a more complicated (empirical) Bayes formulation developed in \citet{martin.mess.walker.eb} and specialized to the present context by \citet{ebcvg}.  Computation of this data-dependent measure requires no Markov chain Monte Carlo (MCMC) or optimization---it has a simple, intuitive, and explicit form.  Most importantly, its theoretical properties are optimal, or nearly so, in every relevant sense:
\begin{itemize}
\item it concentrates asymptotically around the true sparse mean vector $\theta^\star$ at the minimax optimal rate, adaptive to the unknown sparsity level; 
\vspace{-2mm}
\item its mean vector is an asymptotically, adaptively minimax estimator;
\vspace{-2mm}
\item it concentrates asymptotically on a subspace of $\RR^n$ whose dimension is roughly the same as the effective dimension of $\theta^\star$; 
\vspace{-2mm}
\item it consistently selects the non-zero entries of $\theta^\star$; 
\vspace{-2mm}
\item and it provides asymptotically valid uncertainty quantification. 
\end{itemize}
The majority of these properties are familiar, perhaps only the last one requires further explanation.  From this data-dependent measure, one can readily extract credible sets that capture a specified fraction of the distribution's mass.  A question is if that credible set is also asymptotically a confidence set in the sense that it covers the true $\theta^\star$ with probability approximately equal to that same fraction.  This has been a hot topic in the Bayesian literature recently \citep[e.g.,][]{pas.szabo.vaart.uq, szabo.vaart.zanten.2015, rousseau.szabo.2020, castillo.szabo.2020}, and here I establish this property for my simple data-dependent measure.  The ideas and techniques developed in \citet{belitser.ddm} and \citet{belitser.nurushev.uq} will prove to be useful in this regard. 

The remainder of the paper is organized as follows.  Section~\ref{S:background} gives some background on an approach for high-dimensional inference that uses empirical or data-driven priors.  The data-dependent measure investigated here originally arose as a variational approximation to this empirical Bayes-style posterior distribution.  Section~\ref{S:properties} investigates the asymptotic concentration properties of the data-dependent measure, providing justification for the claims made in the bulleted list above.  Some comments about computation of the data-dependent measure are given in Section~\ref{S:example}, along with an illustration to show that its theoretical and computational simplicity don't come at the cost of statistical inefficiency in finite samples.  Concluding remarks are given in Section~\ref{S:discuss} and all the proofs are presented in Appendix~\ref{S:proofs}.

\section{A simple data-dependent measure}
\label{S:background}

\subsection{Empirical prior and its posterior}

While I'm not doing so here in this paper, it is common to assume that the $Y_i$'s are independent and normally distributed, i.e., $Y_i \sim \nm(\theta_i, \sigma^2)$, for $i=1,\ldots,n$, with known variance $\sigma^2$.  This distributional assumption determines a likelihood function and, in turn, a Bayes or empirical Bayes approach can be taken.  One such approach developed recently is that in \citet{martin.walker.eb}, which makes us of an {\em empirical} or {\em data-dependent} prior, that is, a prior distribution that directly depends on the data.  The details here are based on the developments in \citet{martin.mess.walker.eb} for the high-dimensional linear regression problem; see \citet{martin.walker.deb} for some general theory on empirical priors and properties of the corresponding posterior distributions. 

In the approach of \citet{martin.walker.deb}---see \citet{ebcvg} for the sequence model special case---the prior is formulated by first expressing $\theta$ as $(S,\theta_S)$, where $S \subseteq \{1,2,\ldots,n\}$ represents the configuration of zeros and non-zeros and $\theta_S$ the configuration-specific parameters.  Then the (empirical) prior for $\theta$, denoted by $\Pi_n$, with subscript ``$n$'' to indicate data dependence, is specified hierarchically as follows.
\begin{itemize}
\item {\em Marginal prior for $S$:} The the prior on the size/cardinality $|S|$ of $S$ has a mass function $f_n$, and the conditional prior for $S$, given its cardinality $|S|=s$ is uniform over the $\binom{n}{s}$ subsets of size $s$.  Several different forms of $f_n$ are discussed in \citet{ebcvg}, but I'll not need these for what I plan to do here.  
\vspace{-2mm}
\item {\em Conditional prior for $\theta_S$, given $S$:} With the configuration $S$ given, it's determined that $\theta_i=0$ for $i \not\in S$, so only a prior for $\theta_S = \{\theta_i: i \in S\}$ is necessary.  For that prior, the choice in \citet{ebcvg} is 
\[ (\theta_S \mid S) \sim \nm_{|S|}(Y_S, \sigma^2 \gamma^{-1} I_{|S|}), \]
where $\sigma^2$ is the assumed value of the variance in the normal model, $\gamma > 0$ is a prior hyperparameter, and $I_{|S|}$ is the $|S| \times |S|$ identity matrix.  Note that this conditional prior depends on data through the centering around $Y_S = \{Y_i: i \in S\}$.  
\end{itemize} 

With this empirical prior and the normal likelihood, \citet{ebcvg} described construction of a corresponding posterior distribution as 
\[ \Pi^n(d\theta) \propto L_n^\alpha(\theta) \, \Pi_n(d\theta), \]
where $L_n(\theta) = (2\pi\sigma^2)^{-n/2} \exp\{-\frac{1}{2\sigma^2}\|Y-\theta\|^2\}$ is the likelihood function and $\alpha \in (0,1)$ is a fractional power.  An explanation of the role played by $\alpha$ is given in \citet{martin.walker.deb} and the references therein.  Again, I'll not be using this posterior so it's not necessary for me to give an explanation of $\alpha$ here.  

\subsection{An approximation}


Computation of the posterior $\Pi^n$ requires MCMC, but this can be done rather efficiently, as demonstrated in \citet{ebcvg} and \citet{erven.szabo.2020}, compared to the proposed Monte Carlo computations for the horseshoe and other priors. The same is true for the high-dimensional regression version; see, e.g., \citet{martin.mess.walker.eb} and \cite{ebpred}. However, as shown in \citet{ray.szabo.vb} and elsewhere, simple and computationally efficient variational approximations of these high-dimensional posterior are possible.  This inspired \citet{vebreg} to develop a corresponding variational approximation for the empirical prior formulation.  

The jumping off point for \citet{vebreg} was the recognition that the empirical prior described above can be written in a very simple form.  Indeed, the prior assumes that $\theta_1,\ldots,\theta_n$ are independent and the respective marginal distributions are  
\begin{equation}
\label{eq:sas}
\theta_i \sim \lambda_n \, \nm(Y_i, \sigma^2 \gamma^{-1}) + (1-\lambda_n) \, \delta_0, \quad i=1,\ldots,n, 
\end{equation}
where $\lambda_n$ is the prior inclusion probability, which depends on the sample size $n$ but not on the particular index $i$.  It's relatively easy to show that $\lambda_n = n^{-1} \E|S|$, where $\E|S|$ is the prior mean for $|S|$ under $f_n$.  For the two kinds of prior mass function $f_n$ they considered, it follows that, for a constant $a > 0$, 
\begin{equation}
\label{eq:lambda}
\lambda_n = O(n^{-(a+1)}), \quad n \to \infty.
\end{equation}
That $\lambda_n$ is vanishing with $n$---and faster than $n^{-1}$---is consistent with the idea that $\theta$ is believed to be sparse.  For simplicity in what follows, I'll take $\lambda_n = n^{-(1+a)}$.  

When the prior can be expressed in the basic spike-and-slab form \eqref{eq:sas}, it is not too difficult to derive a variational approximation to the posterior distribution.  Indeed, consider a mean-field approximation family of the form 
\[ \bigotimes_{i=1}^n \{ \phi_i \, \nm(\mu_i, \tau_i^2) + (1-\phi_i) \, \delta_0 \}. \]
This corresponds to independent components, each being a mixture of a normal and a point mass at 0, but with component-specific parameters.  The variational approximation proceeds by finding the set of parameters $\{(\mu_i, \tau_i^2, \phi_i): i=1,\ldots,n\}$ that minimizes the Kullback--Leibler divergence of the posterior distribution $\Pi^n$ from the above family.  As \citet{vebreg} show, there are closed-form expressions---no update equations as is typical---for those ``best'' parameters:
\begin{align}
\mu_i & = y_i \notag \\
\tau_i^2 & = \sigma^2 (\alpha + \gamma)^{-1} \label{eq:pars} \\
\logit(\phi_i) & = \logit(\lambda_n) + \tfrac12  \log\tfrac{\gamma}{\alpha+\gamma}+\tfrac{\alpha}{2\sigma^2}y_i^2. \notag 
\end{align}
Here, $\logit(\phi) = \log\{\phi / (1-\phi)\}$.  Note that the variance component $\tau_i^2$ is actually the same for each $i$. Moreover, the weights $\phi_i$ that indicate whether a $\theta_i$ is zero or non-zero are increasing in $y_i^2$, as one would expect.

\subsection{My proposal}

My proposal here in this paper is to define a data-dependent measure 
\begin{equation}
\label{eq:ddm}
\Delta^n = \bigotimes_{i=1}^n \{ \phi_i \, \nm(\mu_i, \tau_i^2) + (1-\phi_i) \, \delta_0 \},  
\end{equation}
with the specific parameters \eqref{eq:pars} plugged in.  This is different from the perspective in \citet{vebreg} because I'm not starting with a normal model, constructing a posterior distribution based on an empirical prior, and then developing a variational approximation.  Instead, I'm directly defining $\Delta^n$ as the data-dependent measure I intend to use for inference on $\theta$.  Consequently, there are no choices about priors to be explained or claims of the variational approximation's accuracy to be justified.  Whether $\Delta^n$ is a reasonable procedure to use rests entirely on what properties it possesses and, as I show next in Section~\ref{S:properties}, its properties are optimal in every practical respect.

\section{Asymptotic properties}
\label{S:properties}

\subsection{Concentration rates}

To set the scene, recall that the errors $Z_1,\ldots,Z_n$ are iid copies of a random variable $Z$ whose distribution is known.  Moreover, I'll assume that $Z$ is subgaussian in the sense that the moment-generating function of $Z$ satisfies 
\begin{equation}
\label{eq:subgaussian}
\E \exp(t Z) \leq \exp(\sigma^2 t^2 / 2), \quad \text{all $t \in \RR$}, 
\end{equation}
where $\sigma$ is a scale parameter, often called the {\em variance proxy}, such that $\sigma^2 \geq \var(Z)$.  Of course, this covers the Gaussian case, but there are other examples too, including bounded random variables; see, e.g., \citet{boucheron.etal.book}.  One key property of subgaussian random variables is that they have exponential tail probability bounds; see Appendix~\ref{SS:prelim} below.  In addition, the tails are sufficiently thin to ensure that the moment-generating function of $(Z/\sigma)^2$ exists in an interval that contains the origin.  What's relevant to the analysis here is the upper endpoint of that interval, which I'll denote as $T > 0$.  The actual endpoint depends on the specific form of the $Z$ distribution, e.g., if $Z$ is Gaussian, then $T=1/2$.  However, the largest bound that I'm aware of that covers all subgaussian cases simultaneously is $T=1/4$ \citep[][Appendix~B]{honorio14}.

As indicated above, interest is in cases where the true mean vector $\theta^\star$ is {\em sparse}, so I need to make this structural assumption precise.  For a generic vector $\theta \in \RR^n$, let $S_\theta$ denote its configuration, i.e., $S_\theta = \{i: \theta_i \neq 0\}$.  Let $|S_\theta|$ denote the cardinality of the configuration.  Of course, $|S_{\theta^\star}| \leq n$ but I have in mind cases where the inequality is strict, even/especially cases where $|S_{\theta^\star}| \ll n$, since these are the only ones in which accurate estimation of $\theta^\star$ is possible.  To characterize this notion of accuracy, recall that the minimax rate \citep[e.g.,][]{donoho1992} relative to $\ell_2$-error is 
\begin{equation}
\label{eq:eps}
\eps_n^2(\theta^\star) = |S_{\theta^\star}| \log(en / |S_{\theta^\star}|). 
\end{equation}
That is, every estimator $\hat\theta$ satisfies 
\[ \sup_{\theta^\star} \eps_n^{-2}(\theta^\star) \E_{\theta^\star} \|\hat\theta - \theta^\star\|^2 \geq 1. \]
So we say that an estimator is minimax optimal if equality holds up to a constant, 
\begin{equation}
\label{eq:estimator}
\sup_{\theta^\star} \eps_n^{-2}(\theta^\star) \E_{\theta^\star} \|\hat\theta - \theta^\star\|^2 \lesssim 1, 
\end{equation}
where ``$\lesssim$'' denotes inequality up to a universal constant. The following two results show that the data-dependent measure $\Delta^n$ in \eqref{eq:ddm}, with suitable choice of $\alpha$, has this same asymptotic concentration rate property.  

\begin{thm}
\label{thm:rate}
For the data-dependent measure $\Delta^n$ in \eqref{eq:ddm}, suppose that $\lambda_n$ and $\alpha$ in \eqref{eq:pars} satisfy, respectively, \eqref{eq:lambda} and $\alpha < 2T$, for $T$ determined by the subgaussian error distribution.  For $\eps_n^2(\theta^\star)$ defined in \eqref{eq:eps} and any sequence $M_n > 0$ with $M_n \to \infty$, 
\[ \sup_{\theta^\star} \E_{\theta^\star} \Delta^n(\{\theta \in \RR^n: \|\theta - \theta^\star\|^2 > M_n \eps_n^2(\theta^\star)\}) \to 0, \quad n \to \infty. \]
\end{thm}

\begin{thm}
\label{thm:mean}
Under the conditions of Theorem~\ref{thm:rate}, the mean vector $\hat\theta$ derived from the data-dependent measure $\Delta^n$ in \eqref{eq:ddm} satisfies \eqref{eq:estimator}.  
\end{thm}

An important observation is that the data-dependent measure $\Delta^n$ is not aware of the sparsity level $|S_{\theta^\star}|$ of the true $\theta^\star$, and yet it concentrates at that specific optimal rate.  This feature is commonly referred to as {\em adaptation}, i.e., the concentration rate of $\Delta^n$ is adaptive to the unknown sparsity level of $\theta^\star$ that determines the optimal rate.  

Of course, the data-dependent measure's (nearly) optimal concentration rate is a plus, but this property alone doesn't imply that $\Delta^n$ is learning the low-dimensional structure in $\theta^\star$.  The next result demonstrates that indeed the data-dependent measure is learning that structure in the sense that the dimension of the space on which $\Delta^n$ concentrates is roughly the same as the effective dimension $|S_{\theta^\star}|$ of $\theta^\star$.  

\begin{thm}
\label{thm:dim}
Under the conditions of Theorem~\ref{thm:rate}, for any $M_n > 1$ such that $M_n \to \infty$, the data-dependent measure $\Delta^n$ in \eqref{eq:ddm} satisfies 
\[ \sup_{\theta^\star}\E_{\theta^\star} \Delta^n(\{\theta \in \RR^n: |S_\theta| > M_n |S_{\theta^\star}|\}) \to 0, \quad n \to \infty. \]
\end{thm}

\subsection{Structure learning}

Theorem~\ref{thm:dim} established that $\Delta^n$ concentrates on a subspace of roughly the effective dimension of $\theta^\star$, but more can be said.  The next result shows that, asymptotically, $\Delta^n$ will not assign positive mass to proper supersets of $S_{\theta^\star}$.  To ensure that all the signals are detectable, an additional assumption about the magnitude of those non-zero $\theta_i^\star$ values is needed.  Specifically, consider 
\begin{equation}
\label{eq:betamin}
\min_{i \in S_{\theta^\star}} |\theta_i^\star| \geq H := \bigl( \tfrac{2\sigma^2 K}{\alpha} \log n \bigr)^{1/2}, \quad \text{for some $K > 2 + a$},
\end{equation}
where $\sigma^2$ is the variance proxy of the error distribution, $\alpha$ is as in \eqref{eq:pars}, and $a$ is as in \eqref{eq:lambda}.  Up to constants, \eqref{eq:betamin} is equivalent to the ``beta-min condition'' common in the high-dimensional estimation literature.  

\begin{thm}
\label{thm:selection}
Under the conditions of Theorem~\ref{thm:rate}, the data-dependent measure $\Delta^n$ satisfies 
\[ \sup_{\theta^\star} \E_{\theta^\star} \Delta^n(\{\theta \in \RR^n: S_\theta \supset S_{\theta^\star}\}) \to 0, \quad n \to \infty. \]
Moreover, if $\theta^\star$ is such that \eqref{eq:betamin} holds, then 
\[ \sup_{\theta^\star} \E_{\theta^\star} \Delta^n(\{\theta \in \RR^n: S_\theta \not\supseteq S_{\theta^\star}\}) \to 0, \quad n \to \infty. \]
If all the above conditions hold, then the two conclusions can be combined, giving 
\begin{equation}
\label{eq:selection}
\E_{\theta^\star} \Delta^n(\{\theta \in \RR^n: S_\theta = S_{\theta^\star}\}) \to 1, \quad n \to \infty. 
\end{equation}
\end{thm}

This theorem says that, asymptotically, $\Delta^n$ will not support configurations that contain zero entries in $\theta^\star$.  Moreover, if the non-zero entries in $\theta^\star$ are sufficiently large, in the sense of \eqref{eq:betamin}, then $\Delta^n$ will not support configurations that miss any of those non-zero entries either.  In the latter case, the only option is that $\Delta^n$ asymptotically supports the true configuration $S_{\theta^\star}$, hence it effectively learns the low-dimensional structure in $\theta^\star$.  The result in \eqref{eq:selection} is often referred to as a {\em selection consistency} property, since any reasonable selection procedure based on $\Delta^n$, e.g., 
\[ \hat S = \arg \max_S \delta^n(S) \quad \text{or} \quad \hat S = \{i: \phi_i > 0.5\}, \]
will, for large enough $n$, identify the correct $S_{\theta^\star}$.

\subsection{Uncertainty quantification}

An important question is if inferences derived from the data-dependent measure are reliable in the sense that they control the frequency of errors, at least asymptotically.  There are a number of ways this can be assessed in the present context.  One is to consider certain one-dimensional summaries of the $n$-dimensional vector $\theta$, in particular, linear functionals.  \citet{vebreg} considered this when treating $\Delta^n$ as a variational approximation of the full (empirical) Bayes posterior under a normal model.  Another angle is to consider a credible set for the full $n$-dimensional vector.  This approach has been considered for a variety of different kinds of Bayes and empirical Bayes posterior distributions in the literature, e.g., \citet{pas.szabo.vaart.uq}, \citet{belitser.ddm}, \citet{belitser.nurushev.uq}, and \citet{belitser.ghosal.ebuq}. Here I'm going to derive analogous properties for the simple data-dependent measure $\Delta^n$ in \eqref{eq:ddm} under the same general subgaussian error structure as above. 

Recall that $\hat\theta$ is the mean vector of the data-dependent measure $\Delta^n$.  Define a ball centered around $\hat\theta$ with radius $r > 0$ as 
\[ B_n(r) = \{\theta \in \RR^n: \|\theta - \hat\theta\| \leq r\}. \]
The goal is to choose a data-dependent radius $r=\hat r$ such that the ball approximately achieves a target coverage probability $1-\zeta$ and has near-optimal size.  There are two natural strategies for selecting the radius $\hat r$ based on $\Delta^n$.  The first is motivated by ensuring the ball has sufficient probability under $\Delta^n$, while the second is motivated by achieving the optimal size.  
\begin{enumerate}
\item {\em Quantile-based}.  Set $\hat r = \inf\{r: \Delta^n(\theta: \|\theta-\hat\theta\| \leq r) \geq 1-\zeta\}$.  
\vspace{-2mm}
\item {\em Plug-in estimator-based}.  Define $\hat S = \{i: \phi_i > \tfrac12\}$ and set $\hat r^2 = |\hat S| \log(en / |\hat S|)$.   
\end{enumerate}
The following theorem summarizes the coverage probability and size properties of the two corresponding credible balls.  

\begin{thm}
\label{thm:uq.ball}
Assume the conditions of Theorem~\ref{thm:selection} hold.  Also, let $\Theta_n \subset \RR^n$ denote the set where the condition \eqref{eq:betamin} on the minimum signal size holds. Fix a significance level $\zeta \in (0,\frac12)$ and a threshold $\eta > 0$.  
\begin{enumerate}
\item Let $\hat r$ denote the $\Delta^n$ quantile-based radius defined above.  Then there exists constants $L$ and $M$ such that, for all sufficiently large $n$,  
\[ \sup_{\theta^\star \in \Theta_n} \prob_{\theta^\star}\{ B_n(M g_n \hat r) \not\ni \theta^\star\} \leq \zeta \quad \text{and} \quad \sup_{\theta^\star} \prob_{\theta^\star}\{ \hat r^2 > L \eps_n^2(\theta^\star)\} \leq \eta, \]
where the inflation factor $g_n$ satisfies $g_n = \log(en)$.  
\item Let $\hat r$ denote the plug-in estimator-based radius defined above.  Then there exists constants $L$ and $M$ such that, for all sufficiently large $n$,  
\[ \sup_{\theta^\star \in \Theta_n} \prob_{\theta^\star}\{ B_n(M \hat r) \not\ni \theta^\star\} \leq \zeta \quad \text{and} \quad \sup_{\theta^\star} \prob_{\theta^\star}\{ \hat r^2 > L \eps_n^2(\theta^\star)\} \leq \eta. \]
\end{enumerate}
\end{thm}

More-or-less explicit expressions for the constants $(L,M)$ are given in the proof, so one could technically use these values for practical implementation.  However, I make no claims that these constants are optimal, in fact, it's likely that they're conservative.  In any case, the point is simply to say that the data-dependent measure spread is ``right'' in the sense that credible balls with slightly larger than optimal size can achieve the nominal coverage probability.  The additional inflation factor $g_n$ in the quantile-based credible ball is needed because, apparently, the quantile itself is too small by a logarithmic factor.  Similar inflation factors have been needed by other authors proving analogous results \citep[e.g.,][]{belitser.ghosal.ebuq}.

\section{Numerical illustration}
\label{S:example}

Computation of the data-dependent measure $\Delta^n$ in \eqref{eq:ddm} is trivial and fast.  Virtually every summary has a closed-form expression, so it's straightforward to produce the mean $\hat\theta$, to select a set of ``active'' variables via $\hat S = \{i: \phi_i > \tfrac12\}$, and to extract marginal credible intervals for each $\theta_i$.  This can be done almost instantaneously, far faster than the computations using the {\tt horseshoe} package in R \citep{horseshoe.package} and at least most of those methods compared in \citet{erven.szabo.2020}.  

There are a host of available methods that provide high-quality estimation and structure learning.  The theoretical support for uncertainty quantification using the simple data-dependent measure is the chief novelty here, so that's what I'll focus on here.  Although the theory presented above is for the joint credible ball, there is good reason \citep[e.g.,][]{vebreg, ebcvg} to believe that the corresponding marginal credible intervals would be approximately valid too.  So my objective in this section is simply to show that the very fast computations do not come at the expense of validity or efficiency.  That is, this simple data-dependent measure produces marginal credible intervals which are as good or better than those from other methods sharing the same theoretical guarantees but with heavier computational burden.  

Specifically, I redo the simulation study presented in \citet{ebcvg} comparing the coverage probability and mean length of the horseshoe and two empirical prior-based credible intervals.  Let $n=500$ and suppose that the errors are iid standard normal.  Similar to  Section~2 of \citet{pas.szabo.vaart.uq}, consider a case where the first five entries of $\theta^\star$ are relatively large, i.e., $\theta_1^\star=\cdots=\theta_5^\star=7$, the second five are intermediate, i.e., $\theta_6^\star = \cdots = \theta_{10}^\star = 2$; $\theta_{11}^\star$ will vary; and the remaining $\theta_{12}^\star,\cdots,\theta_n^\star$ are 0.  Of interest is to see how large $\theta_{11}^\star$ needs to be in order for the coverage probability to be approximately equal to 0.95, the nominal level.  Figure~\ref{fig:cvg} plots the empirical coverage probability and mean lengths of the four marginal credible intervals for $\theta_{11}$, as a function of the signal size $\theta_{11}^\star$.  Of course, the coverage probability will be low when the signal size is small, so of primary interest is how quickly the coverage probability climbs to near 0.95 as $\theta_{11}^\star$ increases.  It's clear that the data-dependent measure and the empirical prior formulation of \citet{martin.walker.eb} perform comparably in the sense that both get to the target coverage probability by around $\theta_{11}^\star \approx 6$, before the other two.  Interestingly, these two are no less efficient in terms of interval lengths, since they're all close to the optimal ``$2 \times 1.96 = 3.92$'' length marked by the horizontal line on the plot.  

Finally, the data-dependent measure can easily scale to far bigger $n$, e.g., $n \sim 10^6$, while the other methods would have serious difficulties with problems of this size.  A thorough comparison of the proposed data-dependent measure with other fast methods for this problem \citep[e.g.,][]{erven.szabo.2020, ray.szabo.vb, rockova.george.2018} would be an interesting direction to pursue.  

\begin{figure}[t]
\begin{center}
\subfigure[Coverage probability]{\scalebox{0.6}{\includegraphics{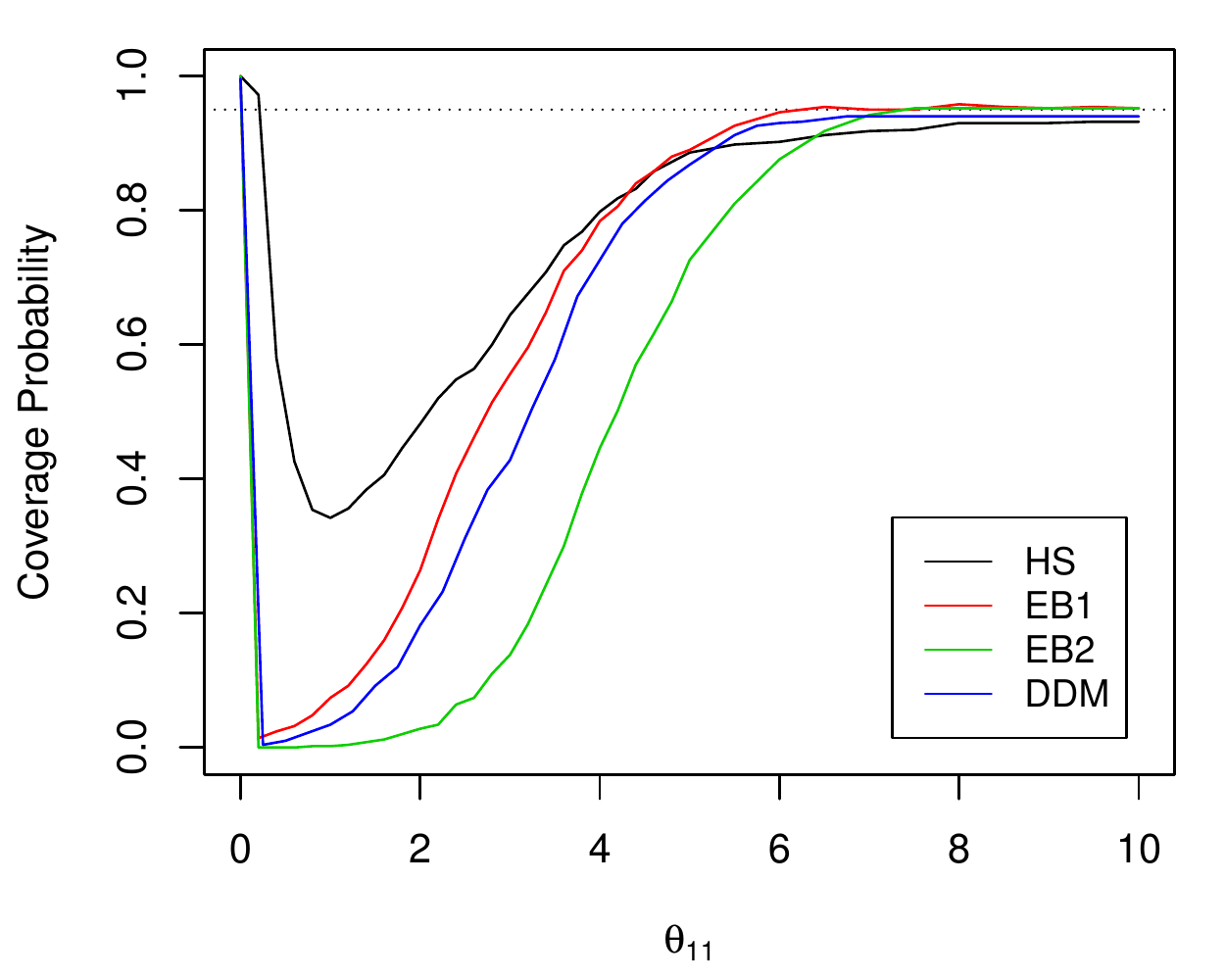}}}
\subfigure[Mean length]{\scalebox{0.6}{\includegraphics{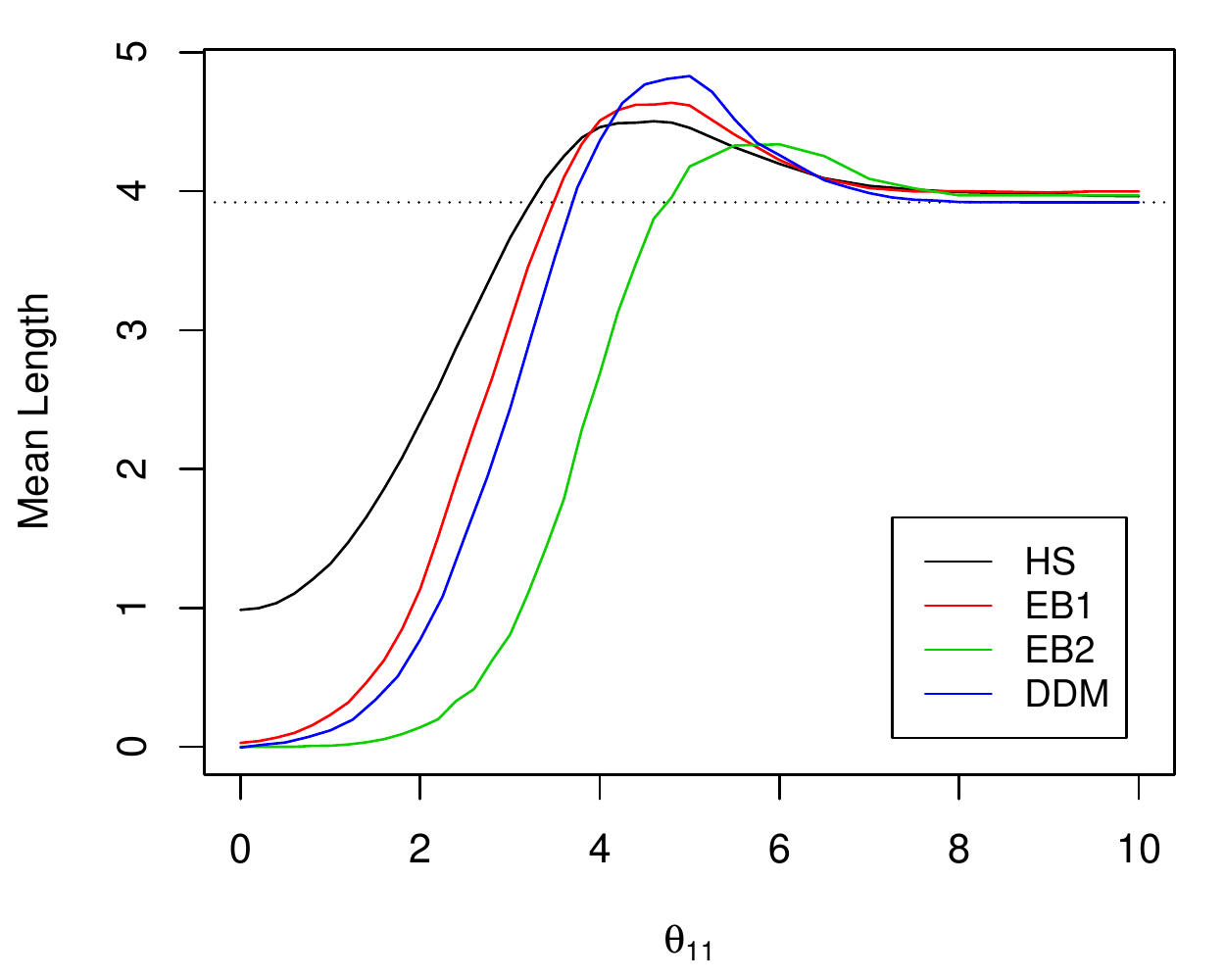}}}
\end{center}
\caption{Plots of the coverage probability and mean length of the marginal credible intervals for $\theta_{11}$, as a function of $\theta_{11}^\star$, based on the four methods: horseshoe (HS), a beta--binomial empirical prior \citep[EB1][]{martin.walker.eb}, a complexity-driven empirical prior \citep[EB2,][]{martin.mess.walker.eb}, and the data-dependent measure (DDM).}
\label{fig:cvg}
\end{figure}

\section{Conclusion}
\label{S:discuss}

In this paper, I've considered inference on a sparse, high-dimensional mean vector using a simple data-dependent measure.  Assuming only subgaussianity of the error distribution, I was able to show that the data-dependent measure has optimal asymptotic convergence properties in virtually every respect.  Most notably, I was able to establish that the data-dependent measure provides asymptotically valid uncertainty quantification in the sense that credible balls centered around the data-dependent measure's mean vector, with suitable data-driven choices of radius, can achieve the nominal coverage probability while maintaining roughly the optimal size.  My proofs of the various results presented herein are relatively straightforward thanks to the simple form of the data-dependent measure under investigation.  Moreover, the simple form makes computation, even for very large-scale problems, fast and easy.  The numerical illustration in Section~\ref{S:example} shows, however, that the theoretical and computational simplicity don't come at the expense of statistical efficiency or poor finite-sample performance.  

One possible extension is to explore the asymptotic behavior of this  data-dependent measure under assumptions on $\theta^\star$ that are mathematically weaker than my notion of sparsity here, e.g., under the so-called {\em excessive bias restriction} in \citet{belitser.ddm}.  I chose to work here with sparsity because it's a simpler and more intuitive condition, but I expect that more general results are possible, even with only subgaussianity assumptions.  

Other kinds of low-dimensional structure in the mean vector can likely be handled using similarly simple data-dependent measures.  For example, in a sequence model where the mean has a piecewise constant structure \citep[e.g.,][]{pas.rockova.2017, ebpiecep}, sparsity shows up in the successive differences, so things would not be too much different from the case considered here.  Clearly there are limitations to how far this kind of simple approach could go, but it's interesting and practically useful to find where the boundary is.  In the high-dimensional regression problem, for example, the theoretical support available in, say, \citet{ray.szabo.vb}, for mean-field variational approximations suggests that other simple and more directly defined data-dependent measures could have similar properties.  It's perhaps not surprising that posterior concentration rate results could be achieved even if the inherent correlation in the full posterior distribution is ignored, but it would be interesting to see if other structure learning or uncertainty quantification properties were similarly unaffected.  

It's a mathematical fact that credible sets in high-dimensions can't be both valid confidence sets and of adaptively optimal size.  The approach that most investigations have taken, including mine here, is to start with a data-dependent measure that achieves the optimal size property and show that its credible sets approximately achieve the target coverage probability too.  To me, the most interesting take-away message from this paper is that apparently very simple solutions can achieve this ``optimal size and approximate coverage'' property.  If both simple and not-so-simple solutions can achieve the same properties, then arguably the standard is too low.  How might we approach these structured high-dimensional problems in a more discriminating way?  One idea is to turn the line of reasoning around, that is, to start with something that achieves valid uncertainty quantification and think about how to introduce the assumed structure, e.g., sparsity, in such a way that efficiency is gained but validity isn't lost.  The approach I have in mind, with developments underway, is to start with a valid inferential model \citep[e.g.,][]{imbasics, imbook}, treat the assumed structure as genuine but incomplete prior information encoded as an imprecise probability, and combine with the inferential model output in an appropriate way that preserves validity.  The main difference between this and the standard approach taken in this paper is that validity is given higher priority than efficiency, which I believe to more appropriate for scientific investigations.



\appendix

\section{Proofs}
\label{S:proofs}

\subsection{Preliminary results}
\label{SS:prelim}

The only distributional assumption being made here is that the errors $Z_1,\ldots,Z_n$ in \eqref{eq:model} are iid copies of a random variable $Z$ with subgaussian tails.  As the name suggests, this condition implies that $Z$ has some Gaussian-like properties.  Here I collect a few relevant facts about subgaussian random variables that will be used in what follows.  

\begin{itemize}
\item It is well known that the square of a subgaussian random variable is subexponential. I won't need any specific properties of subexponential random variables, so there's no need to give a formal definition.  All that matters here is that subexponential random variables have a moment-generating function in an interval that contains the origin.  In particular, 
\begin{equation}
\label{eq:mgf1}
\E e^{t(Z/\sigma)^2} \lesssim 1, \quad t \in (0, T]. 
\end{equation}
Moreover, since translations don't effect the tails of a distribution, the moment-generating function of $(Z+u)^2$, for any $u$, also exists for some arguments, in particular, when the argument is negative, I get 
\begin{equation}
\label{eq:mgf2}
\E e^{-t(Z+u)^2/\sigma^2} \lesssim e^{-t u^2/\sigma^2}, \quad t > 0. 
\end{equation}
\item An equivalent definition of subgaussian random variables is that they admit an exponential tail probability bound just like the Gaussian.  In particular, 
\begin{equation}
\label{eq:tail.prob}
\prob(|Z| > t) \leq 2 e^{-t^2/2\sigma^2}, \quad t > 0. 
\end{equation}
\end{itemize} 

Next are two results that'll be needed in the proofs of the theorems below.  These make use of the properties for subgaussian random variables described above. 

\begin{lem}
\label{lem:sum.phi}
If $\lambda_n$ satisfies \eqref{eq:lambda} and $\alpha < 2T$, then 
\[ \sum_{i=1}^n \E_{\theta_i^\star} \phi_i \leq |S^\star| + o(1) \lesssim |S^\star|. \]
\end{lem}

\begin{proof}
First, split the sum as 
\[ \sum_{i=1}^n \E_{\theta_i^\star} \phi_i = \sum_{i \in S^\star} \E_{\theta_i^\star} \phi_i + \sum_{i \not\in S^\star} \E_{\theta_i^\star} \phi_i. \]
Since $\phi_i \leq 1$, the sum over $i \in S^\star$ is clearly $\leq |S^\star|$.  For the sum over $i \not\in S^\star$, note that all the means are zero and, therefore, all the terms in the sum are the same, i.e.,
\[ \sum_{i \not\in S^\star} \E_{\theta_i^\star} \phi_i = (n - |S^\star|) \E_0 \phi_i. \]
Again, since $\phi_i \leq 1$, it follows that 
\[ \E_0 \phi_i \leq \E_0 e^{\logit(\phi_i)} = \xi_n \, \E e^{\frac{\alpha}{2} (Z/\sigma)^2}, \]
where 
\[ \xi_n = \exp\{\logit(\lambda_n) + \tfrac12 \log\tfrac{\gamma}{\alpha + \gamma}\}. \]
Since $\alpha < 2T$ by assumption, it follows from \eqref{eq:mgf1} that 
\[ \E_0 \phi_i \lesssim \xi_n \lesssim \exp\{\logit(\lambda_n)\} = n^{-(1+a)}, \]
which implies 
\[ \sum_{i \not\in S^\star} \E_{\theta_i^\star} \phi_i \lesssim n^{-a} = o(1), \quad \text{as $n \to \infty$}. \]
Combining this with the bound from the sum over $i \in S^\star$ completes the proof.  
\end{proof}

\begin{lem}
\label{lem:L2}
If the $\lambda_n$ in \eqref{eq:ddm} satisfies \eqref{eq:lambda}, then 
\[ \E_{\theta^\star} \int \|\theta - \theta^\star\|^2 \, \Delta^n(d\theta) \lesssim \eps_n^2(\theta^\star). \]
\end{lem}

\begin{proof}
By the definition of $\Delta^n$ in \eqref{eq:ddm}, it's easy to check that 
\begin{align*}
\int \|\theta - \theta^\star\|^2 \, \Delta^n(d\theta) & = \sum_{i=1}^n \int (\theta_i - \theta_i^\star)^2 \, \Delta^n(d\theta) \\
& = \sum_{i=1}^n \Bigl\{ \phi_i \int (\theta_i - \theta_i^\star)^2 \, \nm(\theta_i \mid Y_i, \tau_i^2) \, d\theta_i + (1-\phi_i) \theta_i^{\star 2} \Bigr\} \\
& = \sum_{i=1}^n \bigl[ \phi_i \{ \tau_i^2 + (Y_i - \theta_i^\star)^2 \} + (1-\phi_i) \theta_i^{\star 2} \bigr] \\
& = \sum_{i=1}^n \tau_i^2 \phi_i + \sum_{i \not\in S^\star} \phi_i Y_i^2 + \sum_{i \in S^\star} \{ \phi_i(Y_i - \theta_i^\star)^2 + (1-\phi_i) \theta_i^{\star 2} \}. 
\end{align*}
Note that $\tau_i^2$ are constant in $i$ and do not depend on data, this can come outside the same (and the following expectation).  Taking expectation, as using the fact that $\phi_i \leq 1$, the right-hand side is bounded by 
\[ \tau^2 \sum_{i=1}^n \E \phi_i + \sum_{i \not\in S^\star} \E \phi_i Y_i^2 + \sigma^2 |S^\star| + \sum_{i \in S^\star} \theta_i^{\star 2} \E(1-\phi_i). \]
I'll deal with each term in this sum separately.  The first term is $\lesssim |S^\star|$ by Lemma~\ref{lem:sum.phi}.  Second, consider $\E_{\theta_i^\star} \phi_i Y_i^2$ for $i \not\in S^\star$, which means $\theta_i^\star = 0$.  For $x > 0$ to be determined, 
\[ \E_0 \phi_i Y_i^2 = \E_0 \phi_i Y_i^2 1_{|Y_i| \leq x} + \E_0 \phi_i Y_i^2 1_{|Y_i| > x}. \]
The first term on the right-hand side is $x^2 \E_0 \phi_i \lesssim x^2 n^{-(a + 1)}$, as shown in the proof of Lemma~\ref{lem:sum.phi}.  The second term is bounded by $\E_0 Y_i^2 1_{|Y_i| > x} = \E Z^2 1_{|Z| > x}$.  For this, we can use the tail probability bound \eqref{eq:tail.prob} for $Z$ as follows:
\begin{align*}
\E Z^2 1_{|Z| > x} & \int_0^\infty \prob(Z^2 1_{|Z| > x} > t) \, dt \\
& = \int_0^\infty \prob\{|Z| > \max(x, t^{1/2})\} \, dt \\
& = \int_0^{x^2} \prob(|Z| > x) \, dt + \int_{x^2}^\infty \prob(|Z| > t^{1/2}) \, dt \\
& \leq 2x^2 e^{-x^2/2\sigma^2} + 2 \int_{x^2}^\infty e^{-t/2\sigma^2} \, dt \\
& \lesssim (x^2 + 1) e^{-x^2/2\sigma^2}.
\end{align*}
Take $x=\{2\sigma^2 \log(n / |S^\star|)\}^{1/2}$, so that 
\[ \E_0 \phi_i Y_i^2 \lesssim n^{-(a+1)} \log(n / |S^\star|) + |S^\star| n^{-1} \log(en / |S^\star|). \]
Summing over $i \not\in S^\star$ gives 
\begin{equation}
\label{eq:second}
\sum_{i \not\in S^\star} \E \phi_i y_i^2 \lesssim n^{-a} \log(n / |S^\star|) + |S^\star| \log(e n / |S^\star|) \lesssim \eps_n^2(\theta^\star). 
\end{equation}
Lastly, for the third term, recall that 
\[ 1 - \phi_i = 1 - \{1 + \xi_n^{-1} e^{-(\alpha/2\sigma^2) Y_i^2}\}^{-1}, \]
where $\xi_n \propto \exp\{-\logit(\lambda_n)\}$. Since $z \mapsto (1 + z)^{-1}$ is convex, Jensen's inequality says 
\[ \E_{\theta_i^\star}(1-\phi_i) \leq 1 - \{1 + \xi_n^{-1} \E_{\theta_i^\star} e^{- (\alpha/2\sigma^2) Y_i^2}\}^{-1}. \]
By \eqref{eq:mgf2}, the expectation satisfies 
\[ \E_{\theta_i^\star} e^{- (\alpha/2\sigma^2) Y_i^2} = \E e^{-(\alpha/2)(Z + \theta_i^\star)^2/\sigma^2} \leq c e^{-(\alpha/2) \theta_i^{\star 2}/\sigma^2}, \]
for a constant $c > 0$.  Therefore, 
\[ \E_{\theta_i^\star}(1-\phi_i) \leq \frac{c\xi_n^{-1} e^{-k \theta_i^{\star 2}}}{1 + c \xi_n^{-1} e^{-k \theta_i^{\star 2}}}, \]
where $k = \alpha/2\sigma^2$.  Multiplying by $\theta_i^{\star 2}$ gives 
\[ \theta_i^{\star 2} \, \E_{\theta_i^\star}(1-\phi_i) \leq \frac{c\xi_n^{-1} \theta_i^{\star 2} e^{-k \theta_i^{\star 2}}}{1 + c \xi_n^{-1} e^{-k \theta_i^{\star 2}}}. \]
As a function of $\theta_i^{\star 2}$, this has the form of a gamma density with shape parameter 2 and rate parameter $k$.  Such a density has mode $k^{-1}$.  Plugging in that mode, what's left is a bounded sequence in $n$, so the right-hand side above is $\lesssim 1$, which implies 
\[ \sum_{i \in S^\star} \theta_i^{\star 2} \, \E_{\theta_i^\star}(1-\phi_i) \lesssim |S^\star|. \]
Putting all the bounds together gives 
\[ \E_{\theta^\star} \int \|\theta - \theta^\star\|^2 \, \Delta^n(d\theta) \lesssim |S^\star| + \eps_n^2(\theta^\star) + 1 \lesssim \eps_n^2(\theta^\star). \qedhere \]
\end{proof}

\subsection{Proofs of Theorems~\ref{thm:rate}--\ref{thm:dim}}

\begin{proof}[Proof of Theorem~\ref{thm:rate}]
By Markov's inequality,
\[ \Delta^n(\{\theta: \|\theta-\theta^\star\|^2 > M_n \eps_n^2(\theta^\star)\}) \leq \frac{1}{M_n \eps_n^2(\theta^\star)} \int \|\theta-\theta^\star\|^2 \, \Delta^n(d\theta). \]
Taking expectation and applying the bound in Lemma~\ref{lem:L2} gives 
\[ \E_{\theta^\star} \Delta^n(\{\theta: \|\theta-\theta^\star\|^2 > M_n \eps_n^2(\theta^\star)\}) \lesssim M_n^{-1}, \]
and since $M_n \to \infty$, the claim follows.  
\end{proof}

\begin{proof}[Proof of Theorem~\ref{thm:mean}]
By Jensen's inequality, 
\[ \|\hat\theta - \theta^\star\|^2 \leq \int \|\theta - \theta^\star\|^2 \, \Delta^n(d\theta). \]
Then the claim follows by Lemma~\ref{lem:L2}.  
\end{proof}

\begin{proof}[Proof of Theorem~\ref{thm:dim}]
By Markov's inequality
\[ \Delta^n(\{\theta: |S_\theta| > M_n |S_{\theta^\star}|\}) \leq \frac{1}{M_n |S_{\theta^\star}|} \sum_{i=1}^n \phi_i, \]
where the sum on the right-hand side is the expectation of $|S_\theta|$ under $\theta \sim \Delta^n$.  Take expectation of both sides and apply the bound in Lemma~\ref{lem:sum.phi}, to get 
\[ \E_{\theta^\star} \Delta^n(\{\theta: |S_\theta| > M_n |S_{\theta^\star}|\}) \lesssim M_n^{-1}, \]
and, since $M_n \to \infty$, the claim follows.  
\end{proof}

\subsection{Proof of Theorem~\ref{thm:selection}}

Let $\delta^n$ denote the mass function of the marginal distribution of $S_\theta$ under $\theta \sim \Delta^n$, i.e., 
\[ \delta^n(S) = \Delta^n(\{\theta: S_\theta = S\}), \quad S \subseteq \{1,2,\ldots,n\}. \]
Also, for the given $\theta^\star$, let $S^\star = S_{\theta^\star}$.  From the simple form of $\Delta^n$, it's easy to check that 
\[ \delta^n(S) = \prod_{i \in S} \phi_i \cdot \prod_{i \not\in S} (1-\phi_i). \]
This leads to a convenient bound
\[ \delta^n(S) \leq \frac{\delta^n(S)}{\delta^n(S^\star)} = \prod_{i \in S \cap S^{\star c}} e^{\logit(\phi_i)} \, \prod_{i \in S^c \cap S^\star} e^{-\logit(\phi_i)}. \]
Since each $\phi_i$ only depends on $Y_i$, and these are independent, we can interchange the order of expectation and product.  Also, for those $i \in S^{\star c}$, with $\theta_i^\star=0$, the $\phi_i$'s are iid, so each term in that product has the same expectation.  Therefore, 
\[ \E_{\theta^\star} \delta^n(S) \leq \bigl\{ \E_0 e^{\logit(\phi_1)} \bigr\}^{|S \cap S^{\star c}|} \, \prod_{i \in S^c \cap S^\star} \E_{\theta_i^\star} e^{-\logit(\phi_i)}. \]
By the moment-generating function bounds in \eqref{eq:mgf1} and \eqref{eq:mgf2}, 
\begin{align*}
\E_0 e^{\logit(\phi_1)} & \leq c_0 \xi_n \\
\E_{\theta_i^\star} e^{-\logit(\phi_i)} & \leq c_1 \xi_n^{-1} e^{-k\theta_i^{\star 2}},
\end{align*}
where $k = \alpha/2\sigma^2$, $\xi_n = \exp\{\logit(\lambda_n)\} \sim n^{-(1+a)}$, and $c_0$ and $c_1$ are the hidden proportionality constants in \eqref{eq:mgf1} and \eqref{eq:mgf2}, respectively.  

Note also that, by Theorem~\ref{thm:dim}, the $\delta^n$-probability of the event ``$|S| > M_n|S^\star|$'' has vanishing expectation for any $M_n \to \infty$.  Thanks to this, I can immediately restrict my attention to those $S$ such that ``$|S| \leq M_n |S^\star|$'' in what follows.  

Now consider two distinct cases separately, namely, $S \supset S^\star$ and $S \not\supseteq S^\star$.  First, for any $S \supset S^\star$, it follows that $|S^c \cap S^\star|=0$.  So
\begin{align*}
\E_{\theta^\star} \delta^n(S: S \supset S^\star) & \leq \sum_{S: S \supset S^\star, |S| \leq M_n|S^\star|} \{ \E_0 e^{\logit(\phi_1)} \}^{|S \cap S^{\star c}|} \\
& = \sum_{t=1}^{M_n|S^\star|} \binom{n-|S^\star|}{t} \{ \E_0 e^{\logit(\phi_1)} \}^t \\
& \leq \sum_{t=1}^{M_n|S^\star|} \{ e(n - |S^\star|)\E_0 e^{\logit(\phi_1)} \}^t \\ 
& \lesssim (n - |S^\star|) \xi_n.
\end{align*}
By definition of $\xi_n$, the upper bound is vanishing as $n \to \infty$, proving the first claim.

Next, for any $S \not\supseteq S^\star$, there must be at least one component in $S^\star$ that is {\em not included} in $S$.  So, 
\begin{align*}
\E_{\theta^\star} \delta^n(S: S \not\supseteq S^\star) & \leq \sum_{S: S \not\supseteq S^\star, |S| \leq M_n|S^\star|} \Bigl[ \bigl\{ \E_0 e^{\logit(\phi_1)} \bigr\}^{|S \cap S^{\star c}|} \, \prod_{i \in S^c \cap S^\star} \E_{\theta_i^\star} e^{-\logit(\phi_i)} \Bigr] \\
& \leq \sum_{S: S \not\supseteq S^\star, |S| \leq M_n|S^\star|} (c_0 \xi_n)^{|S \cap S^{\star c}|} (c_1 \xi_n^{-1} e^{-k H^2})^{|S^c \cap S^\star|} \\
& = \sum_{s=0}^{M_n|S^\star|} \sum_{t=0}^{s \wedge (|S^\star|-1)} \binom{|S^\star|}{t} \binom{n-|S^\star|}{s-t} (c_0 \xi_n)^{s-t} (c_1 \xi_n^{-1} e^{-k H^2})^{|S^\star|-t} \\ 
& \leq \sum_{s=0}^{M_n|S^\star|} \sum_{t=0}^{s \wedge (|S^\star|-1)} \{c_0 (n-|S^\star|)\xi_n\}^{s-t} \{c_1 |S^\star| \xi_n^{-1} e^{-k H^2}\}^{|S^\star|-t}. 
\end{align*}
(In the above derivation, $s$ represents $|S|$ and $t$ represents $|S \cap S^\star|$, which implies $s-t=|S \cap S^{\star c}|$ and $|S^\star|-t = |S^c \cap S^\star|$.)  Note that $t < |S^\star|$ because $S \not\supseteq S^\star$ implies that $S$ can't include all the entries in $S^\star$.  This means that there is a constant factor 
\begin{equation}
\label{eq:common}
|S^\star| \xi_n^{-1} e^{-k H^2}, 
\end{equation}
which goes to 0 as $n \to \infty$ if $H$ is as in \eqref{eq:betamin}.  The terms that involve $(n-|S^\star|) \xi_n$ also vanish as in the previous case above.  So all the terms in the sum are geometrically small, hence the sum is bounded.  But since the common factor \eqref{eq:common} vanishes, the upper bound itself vanishes, proving the second claim of the theorem.  


\subsection{Proof of Theorem~\ref{thm:uq.ball}}

For any data-dependent radius $\hat\rho$, ``$B_n(\hat\rho) \not\ni \theta^\star$'' is equivalent to ``$\|\hat\theta-\theta^\star\| > \hat\rho$,'' and the following decomposition, which holds for any deterministic $R > 0$, is helpful:
\begin{align}
\prob_{\theta^\star}\{\|\hat\theta - \theta^\star\| > \hat \rho\} & \leq \prob_{\theta^\star}\{\|\hat\theta - \theta^\star\| > \hat \rho, \, \hat\rho > R\} + \prob_{\theta^\star}\{\hat \rho \leq R\} \notag \\
& \leq \prob_{\theta^\star}\{\|\hat\theta - \theta^\star\|^2 > R^2\} + \prob_{\theta^\star}\{\hat \rho^2 \leq R^2\}. \label{eq:square}
\end{align}
The first term in \eqref{eq:square} has nothing to do with the radius, so this can be approached the same way for both types of credible balls.  Indeed, by Markov's inequality, 
\[ \prob_{\theta^\star}\{\|\hat\theta - \theta^\star\|^2 > R^2\} \leq R^{-2} \E_{\theta^\star}\|\hat\theta - \theta^\star\|^2. \]
By Theorem~\ref{thm:mean}, the expectation in the upper bound is no more than $M'\eps_n^2(\theta^\star)$ for some constant $M' > 0$ and $\eps_n^2(\theta^\star)$ in \eqref{eq:eps}.  Therefore, if $R^2$ is a suitable multiple of $\eps_n^2(\theta^\star)$, then the first term in \eqref{eq:square} can be made less than a fraction of $\zeta$.  The specific constants depend on how $\hat r$ is defined, and the details for each case are presented below.  

\begin{proof}[Proof for the {\sc quantile-based} radius]
Start with a bound on the non-coverage probability for the  quantile-based radius.  Here I'll bound the non-coverage probability by a sum of three terms, each will be bounded by $\zeta/3$, for sufficiently large $n$.  The first of these three terms comes from the above analysis, so we set $R^2 = (3M'/\zeta) \eps_n^2(\theta^\star)$ and conclude 
\[ \prob_{\theta^\star}\{\|\hat\theta - \theta^\star\|^2 > R^2\} \leq \zeta/3. \]
Next, to bound the second term in \eqref{eq:square}, recall that $\hat\rho^2 = M^2 g_n^2 \hat r^2$, where $\hat r$ is based on the $(1-\zeta)$-quantile of $\Delta^n$ and $M$ is some sufficiently large constant yet to be determined.  Define $\tR = (M g_n)^{-1/2} R$.  Then by definition of $\hat r$, and Markov's inequality (again),  
\begin{align*}
\prob_{\theta^\star}\{\hat\rho \leq R\} & = \prob_{\theta^\star}\{\hat r \leq \tR\} \\
& \leq \prob_{\theta^\star}\{ \Delta^n(\|\theta-\hat\theta\| \leq \tR) \geq 1-\zeta\} \\
& \leq (1-\zeta)^{-1} \E_{\theta^\star}\Delta^n(\|\theta-\hat\theta\| \leq \tR). 
\end{align*}
So the second term in \eqref{eq:square} can be upper-bounded if the above expectation can be upper-bounded.  Towards this, let $S^\star = S_{\theta^\star}$ and use the total probability formula to write 
\begin{align*}
\Delta^n(\|\theta - \hat\theta\| \leq \tR) & = \sum_S  \Delta^n(\|\theta - \hat\theta\| \leq \tR \mid S) \, \delta^n(S) \\
& \leq \{1-\delta^n(S^\star)\} + \Delta^n(\|\theta - \hat\theta\| \leq \tR \mid S^\star).
\end{align*}
Under the conditions of Theorem~\ref{thm:selection}, the first term has expectation that vanishes with $n$, hence is eventually smaller than $\zeta/3$.  So it remains to look at the second term in the above display.  The conditional distribution of $\theta$ under $\Delta^n$, given $S^\star$, is 
\[ \theta \sim \nm_{|S^\star|}(Y_{S^\star}, \tau^2 I) \otimes \delta_{|S^{\star c}|}, \]
where $\tau^2 = \sigma^2 (\alpha + \gamma)^{-1}$ from \eqref{eq:pars} is deterministic.  For such a $\theta$, 
\begin{align*}
\|\theta - \hat\theta\|^2 & \overset{d}{=} \sum_{i \in S^\star} \{\tau G_i + (1-\phi_i)Y_i\}^2 + \sum_{i \not\in S^\star} (\phi_i Y_i)^2 \\
& \geq \sum_{i \in S^\star} \{\tau G_i + (1-\phi_i)Y_i\}^2, 
\end{align*}
where $G_{S^\star} = \{G_i: i \in S^\star\}$ are iid $\nm(0,1)$.  Then 
\begin{align*}
\Delta^n(\|\theta-\hat\theta\| \leq \tR \mid S^\star) & \leq \prob\Bigl\{ \sum_{i \in S^\star} \{\tau G_i + (1-\phi_i)Y_i\}^2 \leq \tR^2 \Bigr\} \\
& \leq \prob\{\tau^2 \|G_{S^\star}\|^2 \leq \tR^2\}, 
\end{align*}
where the second line follows by Anderson's inequality. Note that $\|G_{S^\star}\|^2 \sim \chisq(|S^\star|)$, which means that $\|G_{S^\star}\|^2$ scales like $|S^\star$.  Consider two cases: $|S^\star| = O(1)$ and $|S^\star| \to \infty$.  In the former case, $\|G_{S^\star}\|^2$ is stochastically bounded, so the upper bound in the above display can be made less than $\zeta/3$ if $\tR$ is small or, equivalently, if $M$ is sufficiently large.  For the latter case, with $|S^\star| \to \infty$, the above probability can be bounded as 
\begin{align*}
\|G_{S^\star}\|^2 \leq \tau^{-2} \tR^2 & \iff \|G_{S^\star}\|^2 - |S^\star| \leq \tau^{-2} \tR^2 - |S^\star| \\
& \iff \|G_{S^\star}\|^2 - |S^\star|  \leq -w |S^\star|, 
\end{align*}
where $w = 1 - (\tau |S^\star|)^{-1} \tR^2$.  To see that $w > 0$ or, equivalently, that $\tR^2 < \tau |S^\star|$ for sufficiently large $n$, plug in the definition of $\tR$ to get 
\[ \frac{(3M'/\zeta) \eps_n^2(\theta^\star)}{M g_n} < \tau|S^\star| \iff \frac{\log(en/|S^\star|)}{\log(en)} < \frac{M \tau}{3 M'/\zeta}. \]
Note that the right-most inequality holds if $M > 3M'/(\tau \zeta)$. Therefore, 
\[ \prob\{\tau^2 \|G_{S^\star}\|^2 \leq \tR^2\} \leq \prob\{ \|G_{S^\star}\|^2 - |S^\star|  \leq - w |S^\star| \}. \]
Using a standard tail probability bound
for chi-square random variables \citep[e.g.,][Lemma~1]{massart2000}, the probability in the above display is bounded by 
\[ \prob\{\tau \|G_S\|^2 \leq \tR^2\} \leq \prob\{\|G_{S^\star}\|^2 - |S^\star| \leq -2(|S^\star| x)^{1/2}\} \leq e^{-x}, \]
where $x = w^2 |S^\star|/4$.  Since $|S^\star| \to \infty$, this upper bound will eventually be less than $\zeta/3$ for $M$ as above.  Putting everything together, all three terms upper bounding the non-coverage probability are less than $\zeta/3$ for $n$ sufficiently large.   

Now, for the size of the credible ball.  Write $\eps_n^2 = \eps_n^2(\theta^\star)$.  By definition of $\hat r$, 
\[ \hat r^2 > L \eps_n^2 \iff \Delta^n(\|\theta-\hat\theta\| \leq L \eps_n^2) < 1-\zeta \iff \Delta^n(\|\theta-\hat\theta\| > L \eps_n^2) > \zeta. \]
For the constant $M'$ used above from Theorem~\ref{thm:mean}, and for the specified threshold $\eta > 0$, take $L \geq M'\{1 + (\zeta\eta/2)^{-1/2}\}^2$.  Then the triangle inequality implies 
\[ \Delta^n(\|\theta-\hat\theta\|^2 > L\eps_n^2) \leq \Delta^n(\|\theta-\theta^\star\|^2 > M'\eps_n^2) + 1\{\|\hat\theta - \theta^\star\|^2 > (\zeta\eta/2)^{-1} M' \eps_n^2\}. \]
Using Markov's inequality twice gives 
\begin{align*}
\prob_{\theta^\star}\{\hat r^2 > L \eps_n^2\} & \leq \zeta^{-1} \E_{\theta^\star} \Delta^n(\|\theta-\hat\theta\| > L \eps_n^2) \\
& \leq \zeta^{-1} \Bigl\{ \E_{\theta^\star} \Delta^n(\|\theta-\theta^\star\|^2 > M' \eps_n^2) + \frac{\zeta \eta \E_{\theta^\star} \|\hat\theta - \theta^\star\|^2}{2 M' \eps_n^2} \Bigr\}. 
\end{align*}
The first term in the curly brackets is vanishing as $n \to \infty$ by Theorem~\ref{thm:rate} and, hence, will eventually be less than $\zeta\eta/2$.  By Theorem~\ref{thm:mean} and the definition of $M'$, the second term in the curly brackets is no more than $\zeta\eta/2$.  Therefore, the upper bound is no more than $\eta$, which proves the claim.  
\end{proof}

\begin{proof}[Proof for the {\sc plug-in estimator-based} radius]
Following the same argument as above, let $R^2 = (2M'/\zeta) \eps_n^2(\theta^\star)$. Then Theorem~\ref{thm:mean} implies that the first term in \eqref{eq:square} is no more than $\zeta/2$.  For the second term in \eqref{eq:square}, note that 
\begin{equation}
\label{eq:mono}
\text{$x \mapsto x \log(en / x)$ is increasing on $[0,n]$}.
\end{equation}
This implies that, for sufficiently small $c > 0$, 
\begin{align*}
|\hat S| \log \tfrac{en}{|\hat S|} \leq c |S^\star| \log \tfrac{en}{|S^\star|} & \implies |\hat S| \log\tfrac{en}{|\hat S|} \leq c |S^\star| \log\tfrac{en}{c |S^\star|} \\
& \iff |\hat S| \leq c |S^\star|. 
\end{align*}
Therefore, to get a bound on $\prob_{\theta^\star}\{\hat \rho^2 \leq R^2\}$, with $\hat\rho = M \hat r$, it suffices to bound 
\[ \prob_{\theta^\star}\{|\hat S| \leq c|S^\star|\}, \]
where $c = 2M'/M\zeta$, which can be made small with suitable choice of $M$.  Note that $|\hat S|$ is a sum of independent but non-identically distributed Bernoulli random variables, so its expected value is 
\[ \mu^\star = \E_{\theta^\star}|\hat S| = \sum_{i \in S^\star} \prob_{\theta_i^\star}(\phi_i > \tfrac12) + (n-|S^\star|) \prob_0(\phi_i > \tfrac12). \]
Both lower and upper bounds for $\mu^\star$ are needed.  Towards this, 
\[ \prob_0(\phi_i > \tfrac12) = \prob_0\{\text{logit}(\phi_i) > 0\} = \prob\bigl[(Z/\sigma)^2 > \tfrac{2\sigma^2}{\alpha}\{\logit(\lambda_n) + \tfrac12 \log \tfrac{\gamma}{\alpha + \gamma}\} \bigr]. \]
From the tail probability bound \eqref{eq:tail.prob}, it follows that 
\[ \prob_0(\phi_i > \tfrac12) \lesssim n^{-(1+a)}, \quad i \not\in S^\star,\]
and, therefore, 
\begin{equation}
\label{eq:mu.star}
\mu^\star \leq |S^\star| + n^{-(1+a)}(n-|S^\star|) = \{1 + o(n^{-a})\} |S^\star|. 
\end{equation}
For the lower bound on $\mu^\star$, 
\[ \mu^\star \geq \sum_{i \in S^\star} \prob_{\theta_i^\star}(\phi_i > \tfrac12) \geq |S^\star| \prob_H(\phi_i > \tfrac12), \]
where $H$ is the minimum (non-zero) signal size in \eqref{eq:betamin}.  Then 
\begin{align*}
\prob_H(\phi_i > \tfrac12) & = 1 - \prob_H\{\logit(\phi_i) < 0\} \\
& = 1-\prob_H\bigl[ -\tfrac{\alpha}{2\sigma^2}Y_i^2 > -\{\logit(\lambda_n) + \tfrac12 \log\tfrac{\gamma}{\alpha + \gamma}\} \bigr].
\end{align*}
Apply the exponential function to both sides of the inequality inside $\prob_H(\cdots)$ and then use Markov's inequality, the bound in \eqref{eq:mgf2}, the size of $H$ in \eqref{eq:betamin}, and an argument like that at the end of the proof of Theorem~\ref{thm:selection} to get 
\[ \prob_H\bigl[ -\tfrac{\alpha}{2\sigma^2}Y_i^2 > -\{\logit(\lambda_n) + \tfrac12 \log\tfrac{\gamma}{\alpha + \gamma}\} \bigr] = o(1), \quad n \to \infty. \]
Therefore, $\mu^\star \geq \{1 - o(1)\} |S^\star|$.  Using the standard Chernoff bounds for Bernoulli random variables, we get 
\[ \prob_{\theta^\star}\{|\hat S| \leq c |S^\star|\} = \inf_{t > 0} e^{t|S^\star|} \E_{\theta^\star} e^{-t|\hat S|} = \exp\Bigl\{ c |S^\star| \log \frac{\mu^\star}{c |S^\star|} + c|S^\star| - \mu^\star\Bigl\}. \]
Plug in the lower and upper bounds for $\mu^\star$ to get 
\[ \prob_{\theta^\star}\{|\hat S| \leq c |S^\star|\} \leq \exp\Bigl[ |S^\star| \Bigl\{ c \log \frac{1 + o(n^{-a})}{c} - (1-c) + o(1) \Bigr\} \Bigr]. \]
The term inside $\{\cdots\}$ is negative for $c < 1$ and $n$ sufficiently large, so the right-hand side can be upper-bounded by $\zeta/2$ if $M$ is a sufficiently large multiple of $2M'/\zeta$.  

Finally, for the size of the credible ball, by the same monotonicity property \eqref{eq:mono} used above, it follows that 
\[ \prob_{\theta^\star}\{\hat r > L \eps_n^2\} \leq \prob_{\theta^\star}\{|\hat S| > L |S^\star|\}. \]
By Markov's inequality and the upper bound on $\mu^\star = \E_{\theta^\star}|\hat S|$ in \eqref{eq:mu.star}, 
\[ \prob_{\theta^\star}\{|\hat S| > L |S^\star|\} \leq \frac{\{1 + o(n^{-a})\} |S^\star|}{L|S^\star|}. \]
Therefore, if $L > \eta^{-1}$, then the size condition $\prob_{\theta^\star}\{\hat r^2 > L \eps_n^2\} \leq \eta$ holds.  
\end{proof}

\bibliographystyle{apalike}
\bibliography{/Users/rgmarti3/Dropbox/Research/mybib}

\end{document}